\documentclass[12pt]{amsart}
\usepackage{geometry}   %????????
\usepackage[colorlinks,citecolor = red, linkcolor=blue,hyperindex]{hyperref}
\usepackage{euscript,eufrak,verbatim, mathrsfs}
\usepackage[psamsfonts]{amssymb}
\usepackage{bbm}
\usepackage{graphicx}
 \usepackage{float}
\usepackage{float, tikz}

 \usepackage[all, cmtip]{xy}

\usepackage{upref, xcolor, dsfont}
\usepackage{amsfonts,amsmath,amstext,amsbsy, amsopn,amsthm}
\usepackage{enumerate}

\usepackage{url}

\usepackage{mathtools}
\usepackage{bookmark}

 \usepackage{euscript}
\usepackage{helvet}         % selects\textbf{\textbf{â¢}} Helvetica as sans-serif font
\usepackage{courier}        % selects Courier as typewriter font
\usepackage{type1cm}        % activate if the above 3 fonts are
\usepackage{multicol}        % used for the two-column index
\usepackage[bottom]{footmisc}% places footnotes at page bottom

\newtheorem{theorem}{Theorem}[section]
\newtheorem*{theorem*}{Theorem B}
\newtheorem{lemma}[theorem]{Lemma}

\newtheorem{proposition}[theorem]{Proposition}
\newtheorem{corollary}[theorem]{Corollary}

\newtheorem*{definition*}{Definition}
\newtheorem*{remark*}{Remark}

\newtheorem*{observation*}{Observation}

\newtheorem*{assumption*}{Assumption}
\newtheorem*{question*}{Question}
\newtheorem*{problem*}{Problem}

\theoremstyle{remark}
\newtheorem{remark}[theorem]{Remark}

\newcommand{\R}{\mathbb{R}}
\newcommand{\N}{\mathbb{N}}

\newcommand{\E}{\mathbb{E}}

\newcommand{\TT}{\mathcal{T}}

\newcommand{\pless}{\preccurlyeq}

\newcommand{\an}{\text{\, and \,}}

\begin{document}

\title{Boundedness of  Gaussian random sums on trees}

\author%[authorlabel1]
{Yong Han}
\address%[authorlabel1]
{Yong HAN: College of Mathematics and Statistics, Shenzhen University, Shenzhen 518060, Guangdong, China}
\email{hanyongprobability@gmail.com}

\author%[authorlabel1]
{Yanqi Qiu}
\address%[authorlabel1]
{Yanqi QIU: School of Mathematics and Statistics, Wuhan University, Wuhan 430072, Hubei, China;  Institute of Mathematics, AMSS, Chinese Academy of Sciences, Beijing 100190, China}
\email{yanqi.qiu@hotmail.com}

\author{Zipeng Wang}
\address{Zipeng WANG: College of Mathematics and Statistics, Chongqing University, Chongqing,
401331, China}
\email{zipengwang2012@gmail.com, zipengwang@cqu.edu.cn}

\thanks{Y. Qiu is supported by grants NSFC Y7116335K1,  NSFC 11801547 and NSFC 11688101 of National Natural Science Foundation of China. Z. Wang is supported by NSFC 11601296}

\begin{abstract}
Let $\mathcal{T}$ be a rooted tree endowed with the natural partial order $\pless$.
Let $(Z(v))_{v\in \mathcal{T}}$ be a sequence of independent standard Gaussian random variables and let  $\alpha = (\alpha_k)_{k=1}^\infty$ be a sequence of real numbers with $\sum_{k=1}^\infty \alpha_k^2<\infty$. Set $\alpha_0  =0$ and define a Gaussian process on $\mathcal{T}$ in the following way:
\[
G(\mathcal{T}, \alpha; v): = \sum_{u\pless v} \alpha_{|u|} Z(u), \quad v \in \mathcal{T},
\]
where $|u|$ denotes the graph distance between the vertex $u$ and the root vertex. Under mild assumptions on $\mathcal{T}$, we obtain a necessary and sufficient condition for the almost sure boundedness of the above Gaussian process. Our condition is also necessary and sufficient for the almost sure  uniform convergence of the Gaussian process $G(\mathcal{T}, \alpha; v)$ along all rooted geodesic rays in $\mathcal{T}$.
\end{abstract}


\maketitle

\setcounter{equation}{0}

\section{Introduction}

\subsection{Gaussian random sums on trees}
Let $\TT$ be a rooted tree endowed with the natural partial order  $\pless$ and the graph distance $d(\cdot,\cdot)$, the root vertex will be denoted by $\rho$. We shall always assume that $\TT$ is locally finite and has no leaves.  For any vertex $v\in \TT$, set
$
|v|: = d(\rho, v).
$
In this paper, we study the almost sure boundedness of Gaussian processes on $\TT$ defined in the following way:
\begin{align}\label{def-X-onT}
G(\TT, \alpha; v): = \sum_{u\pless v} \alpha_{|u|} Z(u), \quad v \in \TT,
\end{align}
where $\alpha_0 \equiv 0$, $\alpha = (\alpha_k)_{k=1}^\infty$ is a sequence of real numbers with $\sum_{k=1}^\infty \alpha_k^2<\infty$ and  $(Z(v))_{v\in \TT}$ is a sequence of independent $N(0,1)$ Gaussian random variables.

A necessary and sufficient condition for the almost sure boundedness of the above Gaussian processes was given by Fernique \cite{Fernique-1976} through the  majorizing measures, see also Talagrand \cite{Talagrand-acta} for more general Gaussian processes. The condition obtained by Fernique is however not explicit in terms of the sequence $(\alpha_k)_{k=1}^\infty$. Lifshits and Linde studied  in \cite{Lifshits-TAMS,LW-2011,Russian-EJP}  the necessary conditions and  sufficient conditions for the almost sure boundedness of  such Gaussian processes in terms of certain growth conditions on the sequence $(\alpha_k)_{k=1}^\infty$.  In fact, Lifshits and Linde \cite{Russian-EJP} also investigated the almost sure boundedness of the following more general  Gaussain processes on $\TT$:
\begin{align}\label{def-X-onT-two}
G(\TT, \alpha, \sigma; v):=\sigma(v)\sum_{u\pless v} \alpha(u) Z(u), \quad v \in \TT,
\end{align}
where $\alpha(\cdot)$ and $\sigma(\cdot)$ are real-valued functions. However, in general, there is a gap between the necessary condtions and  the sufficient conditions in \cite{Russian-EJP}. For the special case where  $(\alpha_{k})_{k\geq 1}$ is a non-increasing sequence of positive numbers, their necessary conditions and sufficient conditions coincide.

% Let $\TT$ be a rooted tree whose root vertex is denoted by $\rho\in \TT$.    We endow $\TT$ with the natural partial order  $\pless$ and the graph distance $d(\cdot,\cdot)$. For any vertex $v\in \TT$, if
%\[
%|v|: = d(\rho, v)=n\ge 0,
%\]
%then we will say that the vertex $v$ is in the $n$-th generation.
%The study of the almost sure boundedness of the Gaussian process on tree has
%attract some attention.
%where

We obtain in Theorem \ref{thm-bdd-degree} below an explicit necessary and sufficient condition in terms of the growth condition on the sequence $(\alpha_k)_{k=1}^\infty$ for the almost sure boundedness of the Gaussian process $(G(\TT, \alpha; v))_{v\in \TT}$ on a large class of rooted trees. In particular, for any rooted homogeneous tree, our condition reads as
\[
\sum_{l=0}^\infty\Big( \frac{1}{l+1} \sum_{k=l+1}^\infty \alpha_k^2\Big)^{1/2}<\infty.
\]
  Moreover, the condition obtained in Theorem \ref{thm-bdd-degree} is also necessary and sufficient for the almost sure  uniform convergence of the Gaussian process $(G(\TT, \alpha; v))_{v\in \TT}$  along all rooted geodesic rays.

To deal with the convergence of our Gaussian process on $\TT$ along rooted geodesic rays,  it is convenient to represent them as a sequence of coupled Gaussian processes on the boundary $\partial \TT$ of $\TT$. For this purpose,  let us recall the classical definition  of the boundary $\partial \TT$.
Denote by  $\N=\{1, 2, \cdots\}$ the set of postive integers. A rooted geodesic ray $\xi$ in $\TT$ is a sequence $\xi = (\rho, v_1, v_2, \cdots)$ of vertices of $\TT$ with $v_1\pless v_2\pless\cdots$ and $d(v_n, v_{n+1}) = 1$ for all $n\in \N$. The boundary of $\TT$ is defined by
\[
\partial \TT : = \Big\{\xi\Big|\text{$\xi$ is a rooted geodesic ray in $\TT$}\Big\}.
\]
We endow $\partial \TT$ with the metric
$
d(\xi, \zeta): = 2^{-|\xi \wedge \zeta|}$ for all $\xi, \zeta \in \partial \TT$,  where $\xi\wedge \zeta\in \TT$ is the largest (with repect to the partial order $\pless$) common vertex in the two rays $\xi$ and $\zeta$ and we use the convention $
|\xi\wedge \xi| = \infty$ for all $\xi \in \partial \TT$.

  For any $n\in \N$, define the map $\pi_n: \partial \TT \rightarrow \TT$ by
\[
\partial \TT \ni \xi = (\rho, v_1, v_2, \cdots) \xrightarrow{\,\,\pi_n\,\,} \pi_n(\xi) =  v_n \in \TT.
\]
 Then the Gaussian process defined in \eqref{def-X-onT} can be  represented as
\begin{align}\label{def-X}
X_n(\TT, \alpha; \xi): =  \sum_{k=1}^n \alpha_k Z(\pi_k(\xi)), \quad n \in \N, \xi \in\partial \TT.
\end{align}
We investigate the maxima of these Gaussian processes:
\begin{align}\label{def-maxima}
\begin{split}
M_n(\TT, \alpha) := \sup_{\xi\in \partial \TT}| X_n(\TT, \alpha; \xi)| \an  M^*(\TT, \alpha) := \sup_{n\in \N} \sup_{\xi\in \partial \TT} | X_n(\TT, \alpha; \xi)|.
\end{split}
\end{align}
We shall also need the following definition: for any fixed $\xi \in \partial \TT$, define a series
\begin{align}\label{def-X-proc}
X(\TT, \alpha; \xi): = \sum_{k=1}^\infty \alpha_k Z(\pi_k(\xi)).
\end{align}
A priori, it is not known whether almost surely,  the series \eqref{def-X-proc} converges for all $\xi \in \partial \TT$. In Theorem \ref{thm-bdd-degree} below, for a large class of trees, we give a necessary and sufficient condition for the  almost sure uniform convergence (with respect to $\xi \in \TT$) of the series \eqref{def-X-proc}.

 The process $(M_n(\TT, \alpha))_{n =1}^\infty$ is closely related to the displacements of branching random walks. The simplest case of a branching random walk is described as follows.  One starts a system with one particle at location 0 at time 0. Suppose that $v$ is a particle at location $S_v$ at time $k$, then this particle $v$ dies at time $k+1$ and gives birth to two children $v_1, v_2$. The two children $v_1, v_2$ start to move independently to  new
locations such that the increments are independent of $S_v$ and distributed as $N(0,\alpha_k^2)$. Let $D_n$ denote the collection of all particles at time $n$. Then the maximal displacement at time $n$ is
\[
\widetilde{M}_n^{BRW}(\alpha):=\max_{v\in D_n} S_v.
\]
Similarly, we can also define
\[
M_n^{BRW}(\alpha):=\max_{v\in D_n} |S_v|.
\]
Let $\TT_2$ denote the binary tree.  Clearly, we have the equality in distribution:
\[
M_n(\TT_2, \alpha) \stackrel{d}{=}  M_n^{BRW}(\alpha).
\]
Given $0=s_0<s_1< \cdots < s_m = 1$ and assume that the sequence $(\alpha_k^2)_{k\ge 1}$ is monotone and is given in the following form:
\[
\alpha_k^2 = \alpha_{k'}^2  \quad \text{ for any $k, k' \in [s_{l-1}n, s_l n)$},
\]
then  Fang and Zeitouni \cite{FZ-2012} obtained the  precise asymptotics  of $\widetilde{M}_n^{BRW}(\alpha)$. The case of  general sequence $(\alpha_k)_{k\ge 1}$ was left open in Fang and Zeitouni \cite{FZ-2012}. The results presented in the current paper provide  in particular the asymptotics of $M_n^{BRW}(\alpha)$ up to a universal multiplicative constant for general sequence $(\alpha_k)_{k\ge 1}$.

\subsection{Main results}

For each $v\in \TT$, let
\[
D(v): = \#\Big\{u\in \TT\Big| v\pless u \an  d(u,v)=1\Big\}.
\]

\begin{theorem}\label{thm-bdd-degree}
Let $\TT$ be a rooted tree such that
\begin{align}\label{degree-assumption}
2\le  D_{min}(\TT) = \inf_{v\in \TT} D(v) \le \sup_{v\in \TT} D(v) = D_{max}(\TT)<\infty
\end{align}
and let  $\alpha = (\alpha_k)_{k=1}^\infty$ be a sequence of real numbers with $\sum_{k=1}^\infty \alpha_k^2<\infty$.
Then the following conditions are equivalent:
\begin{itemize}
\item[(1)] $M^*(\TT, \alpha)$ is almost surely bounded;
\item[(2)] the following convergence holds:
\begin{align}\label{eqn-ssufficiency}
\lim_{N\rightarrow\infty}\mathbb{E}\Big[\sup_{n\geq N}\sup_{m\geq 0}\sup_{\xi\in\partial\TT}\Big|\sum_{k=n}^{n+m}\alpha_kZ(\pi_k(\xi))\Big|^2\Big]=0;
\end{align}
\item[(3)] almost surely, the series \eqref{def-X-proc}
converges uniformly in $\xi \in \partial \TT$;
\item[(4)] the sequence $\alpha = (\alpha_k)_{k=1}^\infty$ satisfies
\begin{align}\label{def-Q-alpha}
Q(\alpha): =  \sum_{l=0}^\infty  \Big( \frac{1}{l+1}\sum_{k=l+1}^\infty \alpha_k^2\Big)^{1/2}<\infty.
\end{align}
\end{itemize}
Moreover, under one of the above equivalent conditions,  there exist two numerical constants $c_1, c_2>0$ (independent of $\TT$ and $\alpha$) such that
\begin{align}\label{eqn::equivalent-sup}
c_1  \sqrt{\log D_{min}(\TT)}\cdot Q(\alpha) \le \Big(\E\Big[ \sup_{\xi\in \partial \TT} | X(\TT, \alpha; \xi)|^2\Big]\Big)^{1/2}\le  c_2  \sqrt{\log D_{max}(\TT)} \cdot Q(\alpha)
\end{align}
and
\begin{align}\label{eqn::equivalent-sup-bis}
c_1  \sqrt{\log D_{min}(\TT)}\cdot Q(\alpha) \le (\E[M^*(\TT, \alpha)^2])^{1/2} \le c_2  \sqrt{\log D_{max}(\TT)} \cdot Q(\alpha).
\end{align}
\end{theorem}

\begin{remark}\label{Thm::Threeconditions}
 Consider the following three conditions on a non-negative number sequence  $\alpha = (\alpha_k)_{k\geq 1}$:
\begin{itemize}
  \item[(c-1)] $\lim_{k\rightarrow \infty}\alpha_k=0$ and $\sum_{k=1}^{\infty}k|\alpha_{k+1}-\alpha_k|<\infty$.
  \item[(c-2)]
  $
  Q(\alpha)<\infty.
  $
  \item[(c-3)]
  $
  \sum_{k=1}^{\infty}\alpha_k<\infty.
  $
\end{itemize}
We have (c-1)$\Longrightarrow$ (c-2)$\Longrightarrow$(c-3), see the Appendix for the details. In general, these implications can not be reversed. Indeed,   for all $k\in \N$, define
\begin{align*}
\alpha_k:=\begin{cases}
\frac{1}{n^2}&\text{ if }k=2^n\\
0&\text{ otherwise }
\end{cases}
\an \beta_k:=\begin{cases}
\frac{1}{n^2}&\text{ if }k=2n\\
0&\text{ otherwise }.
\end{cases}
\end{align*}
Then $(\alpha_k)_{k\ge 1}$ satisfies (c-3) but does not satisfy (c-2) and  $(\beta_k)_{k\ge 1}$ satisfies (c-2) but does not satisfy (c-1).  However, when $(\alpha_k)_{k\geq 1}$ is non-increasing, then  all the three conditions (c-1),(c-2) and (c-3) are equivalent.
\end{remark}

\begin{remark}
For any sequence $\alpha = (\alpha_k)_{k=1}^\infty$ of real numbers with $\sum_{k=1}^\infty \alpha_k^2 <\infty$, the kernel $C_\alpha(\cdot, \cdot)$ defined by
\begin{align}\label{def-C-kernel}
C_\alpha(\xi, \zeta)= \sum_{k=1}^{|\xi\wedge \zeta|} \alpha_k^2, \quad \xi, \zeta \in \partial \TT
\end{align}
is non-negative definite and therefore corresponds to a centered Gaussian process
\[
(\widehat{X}(\TT, \alpha; \xi))_{\xi \in \partial \TT}.
\]
By Marcus-Pisier's approach  or Talagrand's approach, under the assumptions \eqref{degree-assumption} and  \eqref{eqn::degree-bound}, the Gaussian process $(\widehat{X}(\TT, \alpha; \xi))_{\xi \in \partial \TT}$ admits a continuous version.  This almost sure conditinuity  implies
\begin{align}\label{X-hat-bdd}
\sup_{\xi\in \partial \TT} |\widehat{X}(\TT, \alpha; \xi)|<\infty, a.s.
\end{align}
Indeed, in the case of binary tree $\TT_2$, following Marcus and Pisier,  we relate the Gaussian process $(\widehat{X}(\TT_2, \alpha; \xi))_{\xi \in \partial \TT}$ to   a random Fourier series defined on the Cantor group $\{-1, 1\}^\N$ which is almost surely uniformly convergent on $\{-1, 1\}^\N$:
\begin{align}\label{intro-X-hat}
\widehat{X}(\alpha; \theta): = \sum_{A \subset \N, \text{$A$ is finite}} a_A(\alpha) g_A \theta_A, \quad \text{$\theta = (\theta_i)_{i=1}^\infty \in \{-1, 1\}^\N$},
\end{align}
where $a_A(\alpha)$ is defined explicitly in terms of $\alpha$ and $(g_A)$ are independent standard Gaussian random variables and $\theta_A$ is the Walsh function:
\begin{align}\label{def-Walsh}
\theta_A = \prod_{i\in A} \theta_i, \quad \theta_\emptyset \equiv 1.
\end{align}

On the other hand, by Theorem \ref{thm-bdd-degree},  under the assumptions \eqref{degree-assumption} and  \eqref{def-Q-alpha}, the Gaussian process $(X(\TT, \alpha; \xi))_{\xi \in \partial \TT}$ has a   covariance kernel also given by \eqref{def-C-kernel}. However,  the uniform convergence of the random Fourier series  \eqref{intro-X-hat}  is different from the following uniform convergence in Theorem \ref{thm-bdd-degree}:
\[
\lim_{n\to\infty} \sup_{\xi \in \partial \TT} |X_n(\TT, \alpha; \xi) - X(\TT, \alpha; \xi)| = 0, a.s.
\]
Note also that, although \eqref{X-hat-bdd} implies that
\[
\sup_{\xi\in \partial \TT} |X(\TT, \alpha; \xi)|<\infty, a.s.,
\]
it is not clear whether this implies
\[
\sup_{n\in \N}\sup_{\xi\in \partial \TT} |X_n(\TT, \alpha; \xi)|<\infty, a.s.
\]
\end{remark}

Theorem \ref{thm-bdd-degree} implies the following

\begin{corollary}
  There exist two numerical constants $c_1, c_2>0$ such that  for any rooted tree $\TT$ satisfying the condition \eqref{degree-assumption} and any finite sequence $\alpha = (\alpha_k)_{k=1}^N$ of real numbers,
\[
c_1  \sqrt{\log D_{min}(\TT)}\cdot Q(\alpha) \le \Big(\E\Big[ \sup_{|v|\le N} G(\TT, \alpha; v)^2\Big]\Big)^{1/2} \le c_2  \sqrt{\log D_{max}(\TT)} \cdot Q(\alpha).
\]
\end{corollary}

Theorem \ref{thm-bdd-degree} combined with Remark \ref{Thm::Threeconditions} implies the following
\begin{corollary}[{Lifshits and Linde \cite[Theorem 2.3]{Russian-EJP}}]\label{Cor:russian}
Let $\TT = \TT_2$ be the binary tree. If the sequence $\alpha=(\alpha_{k})_{k\geq 1}$ is positive and non-increasing, then the Gaussian process \eqref{def-X-onT} is almost surely bounded if and only if
$
\sum_{k=1}^\infty \alpha_k<\infty.
$
\end{corollary}

Moreover, combined with \cite[Proposition 7.1]{Russian-EJP}, Theorem \ref{thm-bdd-degree} also implies the following
\begin{corollary}\label{Cor:two-weight}
 Let  $\varphi:\N\rightarrow \R$ be any function with
\[
\sum_{n=0}^{\infty}\Big(\frac{1}{n+1}\sum_{k=n+1}^{\infty} \varphi(k)^2 \Big)^{1/2}<\infty.
\]
Then for any  functions $\alpha(\cdot)$ and $\sigma(\cdot)$  in the definition \eqref{def-X-onT-two} satisfying
\[
\alpha(v) \sigma(v) = \varphi(|v|) \quad \text{for all $v\in \TT_2$},
\]
the Gaussian process defined by \eqref{def-X-onT-two} on the binary tree $\TT_2$ is almost surely bounded.
\end{corollary}

Finally, we state the result in the case where the degrees of the vertices in our tree is not necessarily uniformly bounded.  For each $n\in \N$, define
\[
D_{min}^{(n)}(\TT): = \min_{v\in \TT \atop |v| = n} D(v), \quad D_{max}^{(n)}(\TT): =  \max_{v\in \TT \atop |v| =n} D(v).
\]

\begin{theorem}\label{thm-all-tree}
Let $\TT$ be a rooted tree such that
\begin{align}\label{eqn::degree-bound}
D_{min}(\TT)\ge 2  \an   \sup_{n\in \N}  \frac{ \log D_{max}^{(n)}(\TT)}{ \log D_{min}^{(n)}(\TT)}<\infty
\end{align}
and let  $\alpha = (\alpha_k)_{k=1}^\infty$ be a sequence of non-negative numbers with $\sum_{k=1}^\infty \alpha_k^2<\infty$.
Then the following conditions are equivalent:
\begin{itemize}
\item[(1)] $M^*(\TT, \alpha)$ is almost surely bounded;
\item[(2)] the following convergence holds:
\[
\lim_{N\rightarrow\infty}\mathbb{E}\Big[\sup_{n\geq N}\sup_{m\geq 0}\sup_{\xi\in\partial\TT}\Big|\sum_{k=n}^{n+m}\alpha_kZ(\pi_k(\xi))\Big|^2\Big]=0;
\]
\item[(3)] almost surely, the series \eqref{def-X-proc}
converges uniformly in $\xi \in \partial \TT$;
\item[(4)] the sequence $\alpha = (\alpha_k)_{k=1}^\infty$ satisfies
\begin{align}\label{eqn::degree-condition-finite}
 \sum_{l=0}^\infty  \Big(\sum_{k=l+1}^\infty \alpha_k^2\Big)^{1/2} \cdot \frac{\log D_{max}^{(l+1)}(\TT)}{\big(\log [\prod_{i=1}^{l+1}D_{max}^{(i)}(\TT)]\big)^{1/2}}<\infty.
\end{align}
\end{itemize}
\end{theorem}

\subsection{Data availability statement}
All data generated or analysed during this study are included in this published article.

\section{The case of the binary tree}

\subsection{The approach of Marcus-Pisier}
We will relate the uniform convergence of the random series \eqref{def-X-proc} to  the theory of uniform convergence of random Fourier series started with the works of Marcus and Pisier \cite{Pisier-book, Pisier-acta}. In this section, we present a suitable version of the main result in \cite{Pisier-book} which will be enough for our applications.

Let $G$ be a compact Abelian group with the Pontryagin dual group $\Gamma = \widehat{G}$, which is countable, see \cite[Chapter 2]{Rudin1962}.
Given a sequence of real numbers $(a_\gamma)_{\gamma\in \Gamma}$ with $\sum_{\gamma\in \Gamma}|a_\gamma|^2<\infty$, we consider a random Fourier series
\begin{align*}\label{E:random-fourier}
F(u)=\sum_{\gamma\in\Gamma}a_\gamma g_\gamma \gamma(u), \quad u\in G,
\end{align*}
where $(g_\gamma)_{\gamma\in \Gamma}$ are independent standard Gaussian random variables.
Let $\mu_G$ be the Haar probability measure on $G$.
For any $u\in G$, let
 \[
  \sigma(u)=\Big(\sum_{\gamma\in\Gamma}|a_\gamma|^2|\gamma(u)-1|^2\Big)^{1/2}.
\]
Recall that the non-decreasing rearrangement $\overline{\sigma}$ of the above function $\sigma$ is defined as follows: first, for any $s>0$, set
\[
m_\sigma(s)=\mu_G(\{u\in G:\sigma(u)<s\}).
\]
Then  $\overline{\sigma}$ is given by
\[
\overline{\sigma}(t)=\sup\{s>0:m_\sigma(s)<t\}, \qquad t\in [0,1].
\]
The entropy integral associated to the function $\sigma$ is defined by
\begin{align}\label{E:pisier-int}
I(\sigma)=&\int_{0}^{1}\frac{\overline{\sigma}(t)}
{t \sqrt{\log(4/t)}}dt \in [0, \infty].
\end{align}
\begin{theorem}[{Marcus-Pisier \cite[Theorem 1.1 \&  Theorem 1.4, Chapter I]{Pisier-book}}]\label{thm-pisier}
Let $G$ be a compact Abelian group with the Pontryagin dual group $\Gamma = \widehat{G}$. Then the random series
\[
\sum_{\gamma\in\Gamma}a_\gamma g_\gamma \gamma(u)
\]
  converges uniformly  almost surely if and only if $I(\sigma)<\infty$.
Furthermore, there exist two  constants $C_1,C_2>0$  (depending only on the group $G$) such that
 \[
C_1\Big[\Big(\sum_{\gamma\in\Gamma}|a_\gamma|^2\Big)^{1/2}+I(\sigma)\Big] \leq \Big( \mathbb{E}  \sup_{u\in G} \Big|\sum_{\gamma\in\Gamma}a_\gamma g_\gamma \gamma(u) \Big|^2 \Big)^{1/2}\leq C_2 \Big[\Big(\sum_{\gamma\in\Gamma}|a_\gamma|^2\Big)^{1/2}+I(\sigma)\Big].
 \]
\end{theorem}

\subsection{Gaussian processes on the boundary of the binary tree}
The rooted binary tree $\TT_2$ is naturally identified with
\begin{align}\label{T-sq-unio}
\TT_2 \simeq \mathcal{D}^{\circ}  = \bigsqcup_{n= 0}^\infty\{-1, 1\}^n,
\end{align}
with convention
$
\{-1, 1\}^0 = \{\rho = \text{the root vertex of $\TT_2$}\}.
$
Using the above identification \eqref{T-sq-unio}, the boundary  $\partial \TT_2$ is naturally identified with
\[
\partial \TT_2 \simeq \mathcal{D}=\{-1,1\}^\mathbb{N}.
\]
For any $\theta\in \mathcal{D}$ and any $k\in \N$, we define
\[
\theta^{(k)}: = (\theta_1, \cdots, \theta_k)\in \{-1,1\}^k \subset \mathcal{D}^{\circ}.
\]

With the product group structure and the product topology, $\mathcal{D} = \{-1,1\}^\mathbb{N}$ is a compact Abelian group, usually called  the Cantor group. The Haar probability measure on $\mathcal{D}$ is denoted by $\mu_{\mathcal{D}}$ and the identity element of the  group $\mathcal{D}$ is denoted by
\begin{align}\label{not-varpi}
\varpi=(1,1,\cdots,1,\cdots).
\end{align}
The Pontryagin dual group $\widehat{\mathcal{D}}$ (see \cite[Chapter 2]{Rudin1962}) of the Cantor group $\mathcal{D}$ is described as follows:  set
\[
\mathcal{F}(\N): = \Big\{A\subset \N\Big| \text{$A$ is finite}\Big\}.
\]
Recall the definitions of the Walsh functions  $\theta_A$ in \eqref{def-Walsh} for all $A\in \mathcal{F}(\N)$.  We have
\[
\widehat{\mathcal{D}} = \Big\{\theta_A\Big|A\in \mathcal{F}(\N)\Big\}.
\]

Let $(Z(v))_{v\in \mathcal{D}^{\circ}}$ be a sequence of independent standard Gaussian random variables.  Let  $(\alpha_k)_{k\geq 1}$ be a sequence of real numbers with
$\sum_{k=1}^\infty \alpha_k^2 <\infty$.  Clearly, for any fixed $\theta \in \mathcal{D}$ or for $\theta$ in a fixed countable subst of $\mathcal{D}$, we can define
\begin{align}\label{E:randomfunction}
X(\alpha; \theta): =\sum_{k=1}^\infty \alpha_k Z(\theta^{(k)}),
\end{align}
where the limit is understood in the sense of $L^2$-convergence. The above definition \eqref{E:randomfunction} does not always give a Gaussian process on $\mathcal{D}$ since  we do not know whether almost surely, the series $X(\alpha; \theta)$ is convergent for all $\theta\in \mathcal{D}$.

For any two distinct $\theta, \eta \in \mathcal{D}$, write
\[
\theta\wedge \eta: = \begin{cases}
\rho &\text{if $\theta_1 \ne \eta_1$}
\\
\theta^{(k)} & \text{if $\theta_i= \eta_i$ for $1\le i \le k$ and $\theta_{k+1} \ne \eta_{k+1}$.}
\end{cases}
\]
Given any number sequence $\alpha = (\alpha_k)_{k=1}^\infty$ with $\sum_{k=1}^\infty \alpha_k^2<\infty$, define a kernel by
\begin{align}\label{def-C-theta}
C_\alpha(\theta, \eta):= \sum_{k  = 0}^{|\theta\wedge \eta|} \alpha_k^2,  \quad \theta, \eta \in \mathcal{D},
\end{align}
with convention $\alpha_0  = 0$ and $|\theta \wedge \theta| = \infty$ for all $\theta \in \mathcal{D}$.

\begin{lemma}\label{lem-pd}
For any number sequence $\alpha = (\alpha_k)_{k=1}^\infty$ with $\sum_{k=1}^\infty \alpha_k^2<\infty$, the kernel $C_\alpha(\cdot, \cdot)$ defined by \eqref{def-C-theta} is non-negative definite.
\end{lemma}

\begin{proof}
It suffices to show that for any finite subset $S\subset \mathcal{D}$, the finite square matrix $C_\alpha|_{S\times S}$ is non-negative definite. Fix any finite subset $S\subset \mathcal{D}$ and note that  the series \eqref{E:randomfunction} is $L^2$-convergent for  all $\theta \in S$. Then  the matrix $C_\alpha|_{S\times S}$ is the covariance matrix of the Gaussian vector $(X(\alpha; \theta))_{\theta \in S}$ and  thus is non-negative definite.
\end{proof}

For any $A\in \mathcal{F}(\N)$, set
\[
m_A = \max \Big\{ n\Big| n \in A \cup \{0\}\Big\}
\]
and
\begin{align}\label{def-a-A}
a_A(\alpha): =\Big(\sum_{k=1}^\infty 2^{-k}\alpha_k^2\mathds{1}(k\geq m_A)\Big)^{1/2}.
\end{align}

\begin{lemma}\label{lem-2-sum}
Let  $(\alpha_k)_{k\geq 1}$ be a number sequence with
$\sum_{k=1}^\infty \alpha_k^2 <\infty$. Then
\begin{align}\label{a-alpha-2sum}
\sum_{A\in \mathcal{F}(\N)} a_A(\alpha)^2 = \sum_{k = 1}^\infty \alpha_k^2.
\end{align}
\end{lemma}
\begin{proof}
Note that for any integer $k\ge 1$, we have
\[
\sum_{A\in \mathcal{F}(\N)}  \mathds{1}(k\ge m_A)=  \sum_{A \subset \{1, 2, \dots, k \}}1 = 2^k.
\]
Therefore,
\begin{align*}
\sum_{A\in \mathcal{F}(\N)} a_A(\alpha)^2 =  & \sum_{A\in \mathcal{F}(\N)} \sum_{k=1}^\infty 2^{-k}\alpha_k^2\mathds{1}(k\geq m_A) = \sum_{k = 1}^\infty 2^{-k} \alpha_k^2  \sum_{A\in \mathcal{F}(\N)}  \mathds{1}(k\ge m_A)
=  \sum_{k=1}^\infty \alpha_k^2.
\end{align*}
This completes the proof of the lemma.
\end{proof}

Let $(g_A)_{A\in \mathcal{F}(\N)}$ be a sequence of independent standard Gaussian random variables. By Lemma \ref{lem-2-sum},  for any fixed $\theta\in \mathcal{D}$ or for $\theta$ in a fixed countable subset of $\mathcal{D}$, we can define
\begin{align}\label{def-X-hat}
\widehat{X}(\alpha; \theta): = \sum_{A\in\mathcal{F}(\N)} a_A(\alpha) g_A \theta_A.
\end{align}
A priori,  we only know that the random variables $\widehat{X}(\alpha; \theta)$ is  well-defined simultaneously for $\theta$ in a fixed countable subset of $\mathcal{D}$.

\begin{lemma}\label{lem-X-and-hat}
For any fixed countable subset $S\subset \mathcal{D}$, the Gaussian process $(\widehat{X}(\alpha; \theta))_{\theta \in S}$ shares the same law as the Gaussian process $(X(\alpha; \theta))_{\theta \in S}$.
\end{lemma}

\begin{proof}
It suffices to show that for any two elements $\theta, \eta \in S \subset \mathcal{D}$, we have
\begin{align}\label{X-X-hat-cov}
 \E[X(\alpha; \theta) X(\alpha; \eta)]  = \E[\widehat{X}(\alpha; \theta) \widehat{X}(\alpha; \eta)] = C_\alpha(\theta, \eta).
\end{align}
Let us only show the second equality. First, for any integer $k\ge 1$, we have
\begin{align*}
\sum_{A\in \mathcal{F}(\N)} \mathds{1}(k\geq m_A)   \theta_A = \sum_{A\subset \{1,2,\dots, k\}}  \theta_A  =  \prod_{i=1}^k (1 + \theta_i)
\end{align*}
and
\[
\prod_{i=1}^k \frac{1+\theta_i \eta_i}{2} = \mathds{1} (\theta^{(k)} = \eta^{(k)}) = \mathds{1}(| \theta\wedge \eta|\ge k).
\]
Therefore,
\begin{align}\label{C-A-sum}
\begin{split}
\sum_{A\in \mathcal{F}(\N)} a_A(\alpha)^2 \theta_A \eta_A &  =\sum_{A\in \mathcal{F}(\N)}  \sum_{k=1}^\infty 2^{-k}\alpha_k^2\mathds{1}(k\geq m_A)   (\theta \eta)_A
\\
& =  \sum_{k=1}^\infty 2^{-k}\alpha_k^2  \sum_{A\in \mathcal{F}(\N)} \mathds{1}(k\geq m_A)   (\theta \eta)_A
\\
&=  \sum_{k=1}^\infty \alpha_k^2 \cdot \prod_{i=1}^k \frac{1 + \theta_i\eta_i}{2} = \sum_{k=1}^\infty \alpha_k^2 \mathds{1}(|\theta\wedge \eta| \ge k)
\\
& = \sum_{k=1}^{|\theta \wedge \eta|}\alpha_k^2 = C_\alpha(\theta, \eta).
\end{split}
\end{align}
It follows that
\begin{align*}
 \E[\widehat{X}(\alpha; \theta) \widehat{X}(\alpha; \eta)]  & =  \E \Big[ \sum_{A\in\mathcal{F}(\N)} a_A(\alpha) g_A \theta_A \sum_{B \in\mathcal{F}(\N)} a_B(\alpha) g_B \eta_B\Big]
\\
 & = \sum_{A\in \mathcal{F}(\N)} a_A(\alpha)^2 \theta_A \eta_A  = C_\alpha(\theta, \eta).
\end{align*}
This completes the proof of the lemma.
\end{proof}

Recall the definition \eqref{def-Q-alpha} for $Q(\alpha)$:
\[
Q(\alpha): = \sum_{l=0}^{\infty}\Big(\frac{1}{l+1} \sum_{k=l+1}^{\infty}\alpha_k^2\Big)^{1/2} \in [0, \infty].
\]

\begin{proposition}\label{T:T2}
Let  $(\alpha_k)_{k\geq 1}$ be a sequence of real numbers with
$\sum_{k=1}^\infty a_k^2 <\infty$. Then  $C_\alpha(\cdot, \cdot)$ is the covariance kernel of a continuous Gaussian process on $\mathcal{D}$
if and only if
\[
Q(\alpha) <\infty.
\]
In such situation, the random series $\widehat{X}(\alpha; \theta)$  defined in  \eqref{def-X-hat} is almost surely convergent uniformly in $\theta\in \mathcal{D}$ and therefore is a realization of a continuous Gaussian process with c covariance kernel $C_\alpha(\cdot, \cdot)$.  Moreover, there exist two numerical constants $c_1, c_2>0$ such that
\begin{align}\label{2-side-2norm}
c_1 Q(\alpha) \le  \sqrt{\E\Big[  \sup_{\theta\in \mathcal{D}} |\widehat{X}(\alpha; \theta)|^2\Big]} \le c_2 Q(\alpha).
\end{align}
\end{proposition}

\begin{remark}
Proposition \ref{T:T2} does not imply directly that the series \eqref{E:randomfunction} is almost surely uniformly convergent in $\theta$.
\end{remark}

\begin{proof}[Proof of Proposition \ref{T:T2}]
Define
\[
\sigma_\alpha(\theta)^2=\sum_{A\in\mathcal{F}(\N)}a_A(\alpha)^{2}
|\theta_A-1|^{2}, \quad \theta \in \mathcal{D}.
\]
By  Theorem \ref{thm-pisier} and Lemma \ref{lem-2-sum}, we only need to show that there exist two numerical constants $c, c'>0$ such that
\begin{align}\label{goal-Q-I}
c Q(\alpha)\le   I(\sigma_\alpha) + \Big(\sum_{A\in \mathcal{F}(\N)} a_A(\alpha)^2\Big)^{1/2}  \le c' Q(\alpha).
\end{align}
Recall the notation \eqref{not-varpi} for $\varpi$.  Using the definition \eqref{def-a-A} for $a_A(\alpha)$, Lemma \ref{lem-2-sum} and the equality \eqref{C-A-sum},  we have
\begin{align*}
\sigma_\alpha(\theta)^2 =& \sum_{A\in\mathcal{F}(\N)}a_A(\alpha)^{2}
|\theta_A-1|^{2}=  2\sum_{A\in\mathcal{F}(\N)} a_A(\alpha)^2(1-\theta_A)\\
=&2\sum_{k=1}^{\infty}\alpha_k^2-2\sum_{k=1}^{\infty}\alpha_k^2\mathds{1}(|\theta\wedge \varpi|\geq k)
=2\sum_{k=1}^{\infty}\alpha_k^2\mathds{1}(|\theta\wedge \varpi|<k)\\
=&2\sum_{k=1}^{\infty}\alpha_k^2\sum_{l=0}^{k-1}\mathds{1}(|\theta\wedge \varpi|=l)
=\sum_{l=0}^{\infty}\mathds{1}(|\theta\wedge \varpi|=l)\cdot\sum_{k=l+1}^{\infty}\alpha_k^2.
\end{align*}
By setting
\begin{align}\label{def-Q-l}
Q_l(\alpha):=\Big(\sum_{k=l+1}^{\infty}\alpha_k^2\Big)^{1/2}, \quad l \ge 0,
\end{align}
we obtain
\[
\sigma_\alpha(\theta)=\sum_{l=0}^{\infty}\mathds{1}(|\theta\wedge \varpi|=l)Q_l(\alpha).
\]
Observe that $Q_l(\alpha)$ is non-increasing in $l$. Let us first assume that $Q_l(\alpha)$ is strictly decreasing in $l$.  For any $s\in(Q_{l+1}(\alpha),Q_l(\alpha)]$, the inequality $\sigma_\alpha(\theta)<s$  holds if and only if $|\theta\wedge \varpi|\geq l+1$. Therefore, for any $s\in (Q_{l+1}(\alpha), Q_l(\alpha)]$, we have
\begin{align*}
m_{\sigma_\alpha}(s) = & \mu_{\mathcal{D}}\Big(\Big\{\theta\in \mathcal{D} \big| \sigma(\theta)<s\Big\}\Big)= \mu_{\mathcal{D}}\Big(\Big\{\theta\in \mathcal{D} \big| | \theta\wedge \varpi | \ge l + 1\Big\}\Big)
\\
= & \mu_{\mathcal{D}}\Big(\Big\{\theta\in \mathcal{D} \big| \theta_1 = \cdots = \theta_l = 1 \an \theta_{l+1} = -1\Big\}\Big)= 2^{-l-1}.
\end{align*}
Thus
\[
m_{\sigma_\alpha}(s) = \mathds{1}(s> Q_0(\alpha)) + \sum_{l=0}^\infty 2^{-l-1} \mathds{1}(Q_{l+1}(\alpha)< s \le Q_l(\alpha)).
\]
It follows that the non-decreasing rearragement $\overline{\sigma_\alpha}$  is given as follows: for any $t\in [0, 1]$, we have
\begin{align*}
\overline{\sigma_\alpha}(t)=\sup \left\{y \mid m_{\sigma_\alpha}(y)<t\right\}
=\sum_{l=0}^{\infty}Q_l(\alpha) \cdot \mathds{1}(2^{-l-1}<t\leq 2^{-l}).
\end{align*}
Now we can compute the associated entropy integral $I(\sigma_\alpha)$:
\begin{align}\label{eq-I-sigma}
I(\sigma_\alpha)=&\int_{0}^{1}\frac{\overline{\sigma}(t)}
{t \sqrt{\log(4/t)}}dt
=\sum_{l=0}^{\infty}Q_l(\alpha)\int_{2^{-l-1}}^{2^{-l}}
\frac{dt}{t\sqrt{\log(4/t)}}
=\sum_{l=0}^{\infty}\frac{\sqrt{\log 2}}{\sqrt{l+3}+\sqrt{l+2}}Q_l(\alpha).
\end{align}
Next, we pass to the general situation where $Q_l(\alpha)$ is not necessarily strictly decreasing. Clearly, we may assume that $S = \{k\in \N: \alpha_k>0\} \ne \emptyset$. Write
\[
S = \{k\in \N: \alpha_k>0\} = \{n_l\}_{l=0}^N,
\]
where $N$ can be either finite or $N= \infty$ and $n_0<n_1<\cdots$. Using the above notation, for any $n_l\le k< n_{l+1}-1$, we have $Q_k(\alpha)= Q_{n_l}(\alpha)$ and hence
\[
\sigma_{\alpha}(\theta)= \sum_{l=0}^{N-1}  Q_{n_l}(\alpha) \mathds{1}( n_l \le  | \theta \wedge \varpi|< n_{l+1}-1).
\]
Since $Q_{n_l}(\alpha)$ is strictly decreasing in $l$, for any $s\in (Q_{n_{l+1}}(\alpha),Q_{n_l}(\alpha)]$,
the inequality $\sigma_\alpha(\theta)<s$  holds if and only if $|\theta\wedge \varpi |\geq n_{l+1}$. Hence for any $Q_{n_{l+1}} (\alpha)< s\leq Q_{n_l}(\alpha)$, we have
\begin{align*}
m_{\sigma_\alpha}(s)=&\mu_{\mathcal{D}}\Big(\theta\in \mathcal{D} | \sigma(\theta)<s\Big)= \mu_{\mathcal{D}}\Big(\Big\{\theta\in \mathcal{D} \big| | \theta\wedge \varpi | \ge n_{l + 1}\Big\}\Big)
\\
= & \mu_{\mathcal{D}}\Big(\Big\{\theta\in \mathcal{D} \big| \theta_1 = \cdots = \theta_{n_{l+1} -1} = 1 \an \theta_{n_{l+1}} = -1\Big\}\Big) = 2^{-n_{l+1}}.
\end{align*}
Consequently, for any $t\in [0, 1)$, we have
\begin{align*}
\overline{\sigma_\alpha}(t)& =\sup \left\{y \mid m_{\sigma_\alpha}(y)<t\right\}
=\sum_{l=0}^{N-1}Q_{n_l} (\alpha) \mathds{1}\Big(2^{-n_{l+1}}<t\leq 2^{-n_l}\Big)
\\
& = \sum_{l=0}^{N-1} \sum_{k= n_l}^{n_{l+1}-1}  Q_k (\alpha)  \mathds{1}(2^{-k-1}<t \le 2^{-k}) = \sum_{k=0}^\infty Q_k (\alpha) \mathds{1}(2^{-k-1}< t \le 2^{-k}).
\end{align*}
Therefore, we get the equality \eqref{eq-I-sigma}.

The equalities  \eqref{a-alpha-2sum} and \eqref{eq-I-sigma} together imply that there exist two constants $c, c'>0$ such that
\begin{align*}
c Q(\alpha) \le I(\sigma_\alpha) + \Big(\sum_{A\in \mathcal{F}(\N)} a_A(\alpha)^2\Big)^{1/2}  = I(\sigma_\alpha) + \Big(\sum_{k=1}^\infty \alpha_k^2\Big)^{1/2} \le c' Q(\alpha).
\end{align*}
We complete the proof of Proposition \ref{T:T2}.
\end{proof}

Our next goal is to show that  the random series \eqref{E:randomfunction} almost surely converges uniformly if and only if $Q(\alpha)<\infty$.

We use the following classical lemma whose proof is included for the reader's convenience.

\begin{lemma}\label{lem-erg}
Let $(\psi(n))_{n\in\mathbb{N}}$ be a non-decreasing sequence of positive numbers such that
\[
\lim_{n\rightarrow \infty}\psi(n)=\infty.
\]
Suppose that  $(a_n)_{n\in\mathbb{N}}$ is a sequence in $\mathbb{C}$ such that the limit
\[
\lim_{N\to\infty}\sum_{n=1}^{N}\frac{a_n}{\psi(n)}
\]
exists. Then
\begin{align*}
\lim_{n\rightarrow\infty} \frac{a_1+a_2+\cdots+a_n}{\psi(n)}=0.
\end{align*}
\end{lemma}
\begin{proof}
For $n\in\mathbb{N}$, define
\[
A_n:=\sum_{k=1}^{n}\frac{a_k}{\psi(k)}.
\]
By convention, we set $A_0:=0$ and $\psi(0):=0$. Using Abel's summation method, we get
\begin{align*}
\sum_{k=1}^{n}a_k&=\sum_{k=1}^{n}\psi(k)(A_k-A_{k-1})
=\psi(n)A_n+\sum_{k=1}^{n-1}(\psi(k)-\psi(k+1))A_k.
\end{align*}
Hence
\begin{align*}
&\frac{a_1+a_2+\cdots+a_n}{\psi(n)}=A_n+\sum_{k=1}^{n-1}\frac{(\psi(k)-\psi(k+1))}{\psi(n)}A_k\\
=&\frac{\sum_{k=1}^{n}(\psi(k)-\psi(k-1))}{\psi(n)}A_n+\frac{\sum_{k=2}^{n}(\psi(k-1)-\psi(k))A_{k-1}}{\psi(n)}\\
=&\frac{\psi(1)}{\psi(n)}A_n+\frac{\sum_{k=2}^{n}(\psi(k)-\psi(k-1))(A_n-A_{k-1})}{\psi(n)}.
\end{align*}
Since $\psi(k)$ is non-decreasing positive,  we get
\begin{align*}
\Big|\frac{a_1+a_2+\cdots+a_n}{\psi(n)}\Big|\leq \frac{\psi(1)}{\psi(n)}|A_n|+\frac{\sum_{k=2}^{n}(\psi(k)-\psi(k-1))}{\psi(n)}|A_n-A_{k-1}|.
\end{align*}
By the assumption that the limit
\[
\lim_{n\rightarrow\infty}A_n = \sum_{k=1}^\infty \frac{a_k}{\psi(k)}
\]
 exists. For any $\epsilon>0$, there exists $N=N(\epsilon)>0$ such that for any $k,l>N$,
\[
|A_k-A_l|<\epsilon.
\]
Let $M=\sup_{n\in\mathbb{N}}|A_n|<\infty$. Then for any $n>N$
\begin{align*}
& \Big|\frac{a_1+a_2+\cdots+a_n}{\psi(n)}\Big|
\\
\leq& \frac{\psi(1)}{\psi(n)}M+\frac{\sum_{k=2}^{N}(\psi(k)-\psi(k-1))}{\psi(n)}2M+\frac{\sum_{k=N+1}^{n}(\psi(k)-\psi(k-1))}{\psi(n)}\epsilon\\
=&\frac{\psi(1)}{\psi(n)}M+2M\frac{\psi(N)-\psi(1)}{\psi(n)}+\frac{\psi(n)-\psi(N+1)}{\psi(n)}\epsilon\\
\leq& \frac{\psi(1)}{\psi(n)}M+2M\frac{\psi(N)-\psi(1)}{\psi(n)}+\epsilon.
\end{align*}
Hence by letting $n\rightarrow\infty$, we get
\begin{align*}
\limsup_{n\rightarrow\infty}\Big|\frac{a_1+a_2+\cdots+a_n}{\psi(n)}\Big|\leq \epsilon.
\end{align*}
Since $\epsilon>0$ is arbitrary, we complete the proof of the lemma.
\end{proof}

\begin{proposition}\label{prop-eq-binary}
 Let $\alpha= (\alpha_k)_{k=1}^\infty$ be a sequence of real numbers with $\sum_{k=1}^\infty \alpha_k^2<\infty$.  The following conditions are equivalent:
\begin{itemize}
\item[(1)] the maximum random variable
\begin{align}\label{def-M-star}
M^*(\alpha): = \sup_{n\in \N} \sup_{\theta \in \mathcal{D}}   \Big| \sum_{k=1}^n \alpha_k Z(\theta^{(k)}) \Big|
\end{align}
 is almost surely bounded;
\item[(2)] the following convergence holds:
\begin{align}\label{eqn-sufficiency}
\lim_{N\rightarrow\infty}\mathbb{E}\Big[\sup_{n\geq N}\sup_{m\geq 0}\sup_{\theta \in \mathcal{D}}\Big|\sum_{k=n}^{n+m}\alpha_kZ(\theta^{(k)})\Big|^2\Big]=0;
\end{align}
\item[(3)] almost surely, the  series $X(\alpha; \theta)$  defined in  \eqref{E:randomfunction}
converges uniformly in $\theta \in \mathcal{D}$;
\item[(4)] the sequence $\alpha = (\alpha_k)_{k=1}^\infty$ satisfies
\[
Q(\alpha)<\infty.
\]
\end{itemize}
Moreover, under one of the above equivalent conditions,  there exist two numerical constants $c_1, c_2>0$ such that
\begin{align}\label{M-star-es}
c_1 \cdot Q(\alpha) \le  \sqrt{\E[\sup_{\theta \in \mathcal{D}} | X(\alpha; \theta)|^2]} \le \sqrt{\E[M^*(\alpha)^2]} \le c_2  \cdot Q(\alpha).
\end{align}
\end{proposition}

\begin{proof}
The implication $(2) \Longrightarrow (3)$ is elementary: by \eqref{eqn-sufficiency} and monotone convergence theorem, we have
\begin{align*}
\mathbb{E}\Big[\lim_{N\rightarrow\infty}\sup_{n\geq N}\sup_{m\geq 0}\sup_{\theta\in\mathcal{D}}\Big|\sum_{k=n}^{n+m}\alpha_kZ(\theta^{(k)})\Big|^2\Big]=0.
\end{align*}
Hence almost surely,
\begin{align*}
\lim_{N\rightarrow\infty}\sup_{n\geq N}\sup_{m\geq 0}\sup_{\theta\in\mathcal{D}}\Big|\sum_{k=n}^{n+m}\alpha_kZ(\theta^{(k)})\Big|=0.
\end{align*}
The almost sure uniform convergence of the random series  $X(\alpha; \theta)$ follows immediately.

The implication $(1) \Longrightarrow (4)$ is given as follows. Assume that
\[
M^*(\alpha) = \sup_{n\in \N} \sup_{\theta \in \mathcal{D}}|\sum_{k=1}^n \alpha_k Z(\theta^{(k)})|<\infty  \quad a.s.
\]
Then, by \cite[Theorem 7.1]{ledouxbook},  $M^*(\alpha)$ is sub-Gaussian.  It follows that
\begin{align}\label{M-star-exp}
\sup_{n\in \N} \E \Big[\sup_{\theta \in \mathcal{D}}\Big|\sum_{k=1}^n \alpha_k Z(\theta^{(k)})\Big|^2\Big]\le \E[M^*(\alpha)^2]<\infty.
\end{align}
But for each fixed $n\in \N$, by defining the finitely supported sequence $\alpha^{(n)} = (\alpha^{(n)}_k)_{k=1}^\infty$ as
\[
\alpha^{(n)}_k: = \alpha_k \mathds{1}(k\le n)
\]
and using \eqref{2-side-2norm}, we have
\begin{align}\label{finite-alpha}
\Big(\E \Big[\sup_{\theta \in \mathcal{D}}\Big|\sum_{k=1}^n \alpha_k Z(\theta^{(k)})\Big|^2\Big]\Big)^{1/2} \ge c_1 Q(\alpha^{(n)}) = c_1 \sum_{l=0}^{n-1} \Big(\frac{1}{l+1} \sum_{k= l+1}^n \alpha_k^2\Big)^{1/2}.
\end{align}
Combining  \eqref{M-star-exp} and \eqref{finite-alpha}, we obtain
\begin{align*}
Q(\alpha) = \sup_{n\in \N} Q(\alpha^{(n)})\le \frac{\sqrt{\E[M^*(\alpha)^2]}}{c_1}<\infty.
\end{align*}

Now we pass to the proof of the implication $(4) \Longrightarrow (2)$. Assume  that $Q(\alpha)<\infty$. We are going to prove \eqref{eqn-sufficiency}. For any $n\in \N$, define  the random variable
\begin{align}\label{def-sigma-n}
\Sigma_n:=\sup_{m\geq 0} \sup_{\theta\in \mathcal{D}}\Big|\sum_{k=n}^{n+m}\alpha_kZ(\theta^{(k)})\Big|
\end{align}
and the sigma-algebra
\[
 \mathcal{F}_n:=\sigma\Big( \Big\{ Z(\theta^{(k)})\Big|  \theta \in \mathcal{D} \an k \ge n \Big\}\Big).
\]
Clearly,
$
\mathcal{F}_1 \supset \mathcal{F}_2\supset \cdots.
$
Notice that for any $n\ge 1$, any $m\ge 0$ and any $\theta\in \mathcal{D}$, we have
\begin{align*}
\mathbb{E}\Big[\sum_{k=n}^{n+m+1}\alpha_kZ(\theta^{(k)})\Big|\mathcal{F}_{n+1}\Big]=\sum_{k=n+1}^{n+m+1}\alpha_kZ(\theta^{(k)}).
\end{align*}
By Jensen's  inequality for conditional expectation, we get
\begin{align*}
\Sigma_{n+1}=&\sup_{m\geq 0} \sup_{\theta\in \mathcal{D}}\Big|\sum_{k=n+1}^{n+1 + m}\alpha_kZ(\theta^{(k)})\Big|=\sup_{m\geq 0} \sup_{\theta\in \mathcal{D}}\Big|\mathbb{E}\Big[\sum_{k=n}^{n+m+1}\alpha_kZ(\theta^{(k)})\Big|\mathcal{F}_{n+1}\Big]\Big|\\
\leq&\mathbb{E}\Big[\sup_{m\geq 0} \sup_{\theta\in \mathcal{D}}\big|\sum_{k=n}^{n+m+1}\alpha_kZ(\theta^{(k)})\big|\Big|\mathcal{F}_{n+1}\Big]\\
=&\mathbb{E}\Big[\sup_{m\geq 1} \sup_{\theta\in \mathcal{D}}\big|\sum_{k=n}^{n+m}\alpha_kZ(\theta^{(k)})\big|\Big|\mathcal{F}_{n+1}\Big]\\
\leq &\mathbb{E}\Big[\sup_{m\geq 0} \sup_{\theta \in \mathcal{D}}\big|\sum_{k=n}^{n+m}\alpha_kZ(\theta^{(k)})\big|\Big|\mathcal{F}_{n+1}\Big]\\
=&\mathbb{E}[\Sigma_n|\mathcal{F}_{n+1}].
\end{align*}
Thus $(\Sigma_n, \mathcal{F}_n)_{n\in\mathbb{N}}$ is a reverse sub-martingale. By Doob's inequality,  for any $N,\tilde{N}\in\mathbb{N}$,
\[
\mathbb{E}[\sup_{N\leq n\leq N+\tilde{N}}\Sigma_n^2]\leq  4 \mathbb{E}[\Sigma_N^2].
\]
Letting $\tilde{N}\rightarrow\infty$, we get that for any $N\in\mathbb{N}$,
\[
\mathbb{E}[\sup_{n\geq N}\Sigma_n^2]\leq  4 \mathbb{E}[\Sigma_N^2].
\]
To investigate $\mathbb{E}[\Sigma_N^2]$, we define for any fixed $n\in\mathbb{N}$,
\[
\widehat{\Sigma}_m^{(n)}:=\sup_{\theta\in \mathcal{D}}\Big|\sum_{k=n}^{n+m}\alpha_kZ(\theta^{(k)})\Big|,\quad m=0, 1,2,...
\]
 and the sigma-algebras
\[
\widehat{\mathcal{F}}_m^{(n)}:=\sigma\Big( \Big\{ Z(\theta^{(k)})\Big|  \theta \in \mathcal{D}, k \leq n+m)\Big\} \Big),\quad m=0,1,2,....
\]
Clearly, we have
$
\widehat{\mathcal{F}}_0^{(n)} \subset \widehat{\mathcal{F}}_1^{(n)} \subset \widehat{\mathcal{F}}_2^{(n)} \subset \cdots
$
and
\[
\mathbb{E}\Big[\sum_{k=n}^{n+m+1}\alpha_kZ(\theta^{(k)})|\widehat{\mathcal{F}}_m^{(n)}\Big]=\sum_{k=n}^{n+m}\alpha_kZ(\theta^{(k)}).
\]
Hence by Jensen's inequality,
\begin{align*}
\widehat{\Sigma}_m^{(n)}=&\sup_{\theta\in \mathcal{D}}\big|\sum_{k=n}^{n+m}\alpha_kZ(\theta^{(k)})\big|=\sup_{\theta\in \mathcal{D}}\big|\mathbb{E}\Big[\sum_{k=n}^{n+m+1}\alpha_kZ(\theta^{(k)})|\widehat{\mathcal{F}}_m^{(n)}\Big]\big|\\
\leq &\mathbb{E}\Big[\sup_{\theta\in \mathcal{D}}\big|\sum_{k=n}^{n+m+1}\alpha_kZ(\theta^{(k)})\big|\Big|\widehat{\mathcal{F}}_m^{(n)}\Big]=\mathbb{E}\big[\widehat{\Sigma}_{m+1}^{(n)}|\widehat{\mathcal{F}}_m^{(n)}\big].
\end{align*}
Thus $(\widehat{\Sigma}_m^{(n)},\widehat{\mathcal{F}}_m^{(n)})_{m\in\mathbb{N}}$ is a sub-martingale. By Doob's inequality again, we have
\[
\mathbb{E}[\Sigma_N^2]=\mathbb{E}[\sup_{m\geq 0}\big(\widehat{\Sigma}_m^{(N)}\big)^2]=\lim_{L\to \infty}\mathbb{E}[\sup_{0\leq m\leq L}\big(\widehat{\Sigma}_m^{(N)}\big)^2]\leq 4 \lim_{L\to \infty}\mathbb{E}[\big(\widehat{\Sigma}_L^{(N)}\big)^2].
\]
Now we study $\mathbb{E}[\big(\widehat{\Sigma}_L^{(N)}\big)^2]$ for fixed $N$ and $L$. Clearly,  by setting $\alpha^{(N,L)}$ as
\[
\alpha^{(N,L)}_k =  \mathds{1}(N\le k\le N+L) \alpha_k,
\]
we have
\[
\widehat{\Sigma}_L^{(N)} =\sup_{\theta\in \mathcal{D}}\Big|\sum_{k=N}^{N+L}\alpha_kZ(\theta^{(k)})\Big| =  \sup_{\theta\in \mathcal{D}}\Big|\sum_{k=1}^{\infty}\alpha_k^{(N,L)} Z(\theta^{(k)})\Big|  =  \sup_{\theta\in \mathcal{D}} |X(\alpha^{(N,L)}; \theta)|.
\]
Therefore, by  \eqref{2-side-2norm}, we have
\begin{align*}
\mathbb{E}\Big[\big(\widehat{\Sigma}_L^{(N)}\big)^2\Big]&=\mathbb{E}\Big[  \sup_{\theta\in \mathcal{D}} |X(\alpha^{(N,L)}; \theta)|^2\Big] \le c_2^2 \cdot  Q(\alpha^{(N,L)})^2
\\
& \leq  c_2^2 \Big[\sum_{l=0}^{N-1}\frac{1}{\sqrt{l+1}}\Big(\sum_{k=N}^{N+L}\alpha_k^2\Big)^{1/2}+\sum_{l=N}^{N+L}\frac{1}{\sqrt{l+1}}\Big(\sum_{k=l+1}^{N+L}\alpha_k^2\Big)^{1/2}\Big]^2.
\end{align*}
Recall the definition  \eqref{def-Q-l} of $Q_l(\alpha)$. We have
\begin{align}\label{sigma-N-es}
\begin{split}
\mathbb{E}[\Sigma_N^2] & \le 4 \lim_{L\to\infty} \mathbb{E}\Big[\big(\widehat{\Sigma}_L^{(N)}\big)^2\Big]
\\
& \leq  4 c_2^2 \Big[\sum_{l=0}^{N-1}\frac{1}{\sqrt{l+1}}\Big(\sum_{k=N}^{\infty}\alpha_k^2\Big)^{1/2}+\sum_{l=N}^{\infty}\frac{1}{\sqrt{l+1}}\Big(\sum_{k=l+1}^{\infty}\alpha_k^2\Big)^{1/2}\Big]^2
\\
&  =  4 c_2^2 \Big[ Q_{N-1}(\alpha) \sum_{l=0}^{N-1}\frac{1}{\sqrt{l+1}} +  \sum_{l=N}^{\infty}\frac{1}{\sqrt{l+1}} Q_l(\alpha)\Big]^2
\\
& \le 4 c_2^2 \Big[2 \sqrt{N} Q_{N-1}(\alpha) + \sum_{l=N}^{\infty}\frac{1}{\sqrt{l+1}} Q_l(\alpha) \Big]^2.
\end{split}
\end{align}
The condition
\[
Q(\alpha) = \sum_{l=0}^\infty \frac{1}{\sqrt{l+1}}Q_l(\alpha)<\infty
\] on the one hand implies
\[
\lim_{N\to\infty} \sum_{l=N}^{\infty}\frac{1}{\sqrt{l+1}} Q_l(\alpha) = 0
\]
and, on the other hand,  by Lemma \ref{lem-erg}, implies
\[
\lim_{N\to\infty} \frac{Q_0(\alpha) + \cdots + Q_{N-1}(\alpha)}{\sqrt{N}} =0.
\]
Hence
\begin{align}\label{N-Q-N}
\limsup_{N\to\infty} \sqrt{N} Q_{N-1}(\alpha) \le \limsup_{N\to\infty} \frac{Q_0(\alpha) + \cdots + Q_{N-1}(\alpha)}{\sqrt{N}} = 0.
\end{align}
We then obtain
\begin{align*}
 \lim_{N\rightarrow\infty}\mathbb{E}\Big[\sup_{n\geq N}\sup_{m\geq 0}\sup_{\theta\in\mathcal{D}}\Big|\sum_{k=n}^{n+m}\alpha_kZ(\theta^{(k)})\Big|^2\Big]  =  \lim_{N\rightarrow\infty} \mathbb{E}[\sup_{n\geq N}\Sigma_n^2]\leq  4 \limsup_{N\to\infty} \mathbb{E}[\Sigma_N^2] = 0.
\end{align*}

The implication $(3) \Longrightarrow (4)$ follows from Lemma \ref{lem-X-and-hat} and Proposition \ref{T:T2}.

We now prove the implication $(4) \Longrightarrow (1)$.  Suppose that $Q(\alpha)<\infty$.  From the definition \eqref{def-sigma-n}, we have
$M^*(\alpha) = \Sigma_1$.
Applying the inequality \eqref{sigma-N-es} for $N=1$, we obtain
\begin{align}\label{goal-M-star}
\E[ M^*(\alpha)^2] \le 16c_2^2  \cdot Q(\alpha)^2.
\end{align}
Hence $M^*(\alpha)<\infty$ a.s.

Finally, the inequalities \eqref{M-star-es} follow from Proposition \ref{T:T2}, the inequality \eqref{goal-M-star} and the following  observation: if almost surely the series \eqref{E:randomfunction}  converges uniformly, then
\[
\sup_{\theta\in \mathcal{D}} | X(\alpha; \theta)| \le M^*(\alpha).
\]
We complete the proof of the proposition.
\end{proof}

\section{More general trees}

In this section, we prove Theorem~\ref{thm-bdd-degree} and Theorem~\ref{thm-all-tree}.  The main ingredient in our proofs for  more general trees  is a reduction to the case of the binary tree.

\subsection{Reduction to the case of the binary tree}

Recall the definitions  \eqref{def-X} and \eqref{def-maxima}. For any rooted tree $\TT$, we define a random variable
\[
\Sigma_N(\TT, \alpha): = \sup_{n\geq N}\sup_{m\geq 0}\sup_{\xi\in\partial\TT}\Big|\sum_{k=n}^{n+m}\alpha_kZ(\pi_k(\xi))\Big|.
\]
Note that if $\iota: \TT\rightarrow \TT'$ is  a root-preserving and partial order-preserving isometric embedding, then $\iota$ induces a natural embedding of $\partial \TT$ into $\partial \TT'$.

We need the following  elementary lemmas whose proofs are omitted.
\begin{lemma}\label{lem-dominance}
Let $\TT$ and $\TT'$ be two rooted trees. Assume that there exists a root-preserving and partial order-preserving isometric embedding $\iota: \TT\rightarrow \TT'$. Then
\[
M_n(\TT, \alpha)\le_s M_n(\TT', \alpha), \quad M^*(\TT, \alpha)\le_s M^*(\TT', \alpha)
\]
and
\[
\Sigma_N(\TT, \alpha) \le_s \Sigma_N(\TT', \alpha) \quad \text{for all $N\in \N$,}
\]
where $\le_s$ means the stochastic domination.
\end{lemma}

Given any sequence $\mathfrak{q} = (q_n)_{n=1}^\infty$ in $\N$, let $\mathcal{T}(\mathfrak{q})$ be the rooted tree such that each vertex in the $(n-1)$-th generation has exactly $q_{n}$ children for each $n\in\N$. Given a sequence $\mathfrak{l} = (l_1, l_2, \dots)$ of positive integers, we define $2^\mathfrak{l}$ to be the following sequence
\begin{align}\label{def-2-l}
2^\mathfrak{l}: = (2^{l_1}, 2^{l_2}, \dots).
\end{align}
For any number sequence $\alpha = (\alpha_1, \alpha_2, \dots)$, let  $\alpha[\mathfrak{l}] = (\beta_1, \beta_2, \dots)$ be the sequence with
\begin{align}\label{def-beta}
\beta_l = \alpha[\mathfrak{l}]_l=\begin{cases}
\alpha_k&\text{ if } l=l_1+l_2+...+l_k\\
0& \text{ otherwise}
\end{cases}.
\end{align}
Recall the definition \eqref{def-Q-alpha} of $Q(\alpha)$ and definition \eqref{def-Q-l} of $Q_l(\alpha)$. Setting $l_0 =0$, we have
\[
Q(\alpha[\mathfrak{l}]) = \sum_{k=0}^\infty Q_k(\alpha) \sum_{l_0 + \cdots + l_k\le l < l_0 + \cdots + l_{k+1}} \frac{1}{\sqrt{l+1}}.
\]
Hence there exist two numerical constants $c, c'>0$ such that
\[
c Q(\alpha[\mathfrak{l}]) \le   \sum_{k=0}^\infty \frac{l_{k+1}}{\sqrt{l_1 + \cdots + l_{k+1}}}\Big(\sum_{n=k+1}^\infty\alpha_n^2\Big)^{1/2} \le c' Q(\alpha[\mathfrak{l}]).
\]

\begin{lemma}\label{lem-dya-subtree}
We have the following equalities in distribution:
\begin{align*}
 M_{n}(\TT(2^{\mathfrak{l}}), \alpha) \stackrel{d}{=}  M_{l_1 + \cdots + l_n}(\TT_2, \alpha[\mathfrak{l}]), \quad
   M^*(\TT(2^{\mathfrak{l}}), \alpha)  \stackrel{d}{=}  M^*(\TT_2, \alpha[\mathfrak{l}])
\end{align*}
and
\[
\Sigma_{n}(\TT(2^{\mathfrak{l}}); \alpha) \stackrel{d}{=}  \Sigma_{l_1 + \cdots+l_n} (\TT_2; \alpha[\mathfrak{l}]).
\]
\end{lemma}

\begin{lemma}\label{lem-two-embed}
Given $\mathfrak{q}=(q_1,q_2,...)$ with $q_i\geq 2$ for any $i\in\N$. There exist two root-preserving and partial order-preserving isometric embeddings
\[\TT(2^{\mathfrak{l}^{-}(\mathfrak{q})}) \hookrightarrow \TT(\mathfrak{q})\hookrightarrow \TT(2^{\mathfrak{l}^{+}(\mathfrak{q})})
\] with $\mathfrak{l}^{\pm}(\mathfrak{q})$ the sequences of positive integers defined by
\[
\mathfrak{l}^{-}(\mathfrak{q}): = \Big(\Big\lfloor  \frac{\log q_1}{\log 2} \Big\rfloor, \Big\lfloor  \frac{\log q_2}{\log 2} \Big\rfloor, \dots \Big) \an \mathfrak{l}^{+}(\mathfrak{q}): = \Big(\Big\lceil  \frac{\log q_1}{\log 2} \Big\rceil, \Big\lceil  \frac{\log q_2}{\log 2} \Big\rceil, \dots \Big),
\]
where $\lfloor x \rfloor$  is the largest integer not bigger than $x$ and $\lceil x \rceil$  is the smallest integer not smaller than $x$.
\end{lemma}

Combining Lemmas \ref{lem-dominance}, \ref{lem-dya-subtree} and \ref{lem-two-embed}, we obtain
\begin{proposition}\label{prop-comp-binary}
Given $\mathfrak{q}=(q_1,q_2,...)$ with $q_i\geq 2$ for any $i\in\N$.  Let $\alpha= (\alpha_k)_{k=1}^\infty$ be a sequence of real numbers with $\sum_{k=1}^\infty \alpha_k^2<\infty$.  The following conditions are equivalent:
\begin{itemize}
\item[(4)]  the random variable
\[
M^*(\TT(\mathfrak{q}), \alpha) = \sup_{n\in \N} \sup_{\xi\in \partial \TT(\mathfrak{q})} | X_n(\TT, \alpha; \xi)|
\] is almost surely bounded;
\item[(2)] the following convergence holds:
\begin{align}\label{T-q-unif}
\lim_{N\rightarrow\infty}\mathbb{E}\Big[\sup_{n\geq N}\sup_{m\geq 0}\sup_{\xi\in\partial\TT(\mathfrak{q})}\Big|\sum_{k=n}^{n+m}\alpha_kZ(\pi_k(\xi))\Big|^2\Big]=0;
\end{align}
\item[(3)] almost surely, the series \eqref{def-X-proc}
converges uniformly in $\xi \in \partial \TT(\mathfrak{q})$;
\item[(1)] the sequence $\alpha = (\alpha_k)_{k=1}^\infty$ satisfies
\begin{align}\label{def-Q-hat}
Q(\mathfrak{q}; \alpha): =   \sum_{k=0}^\infty \frac{\log q_{k+1}}{\sqrt{\log (q_1 \cdots q_{k+1})}}\Big(\sum_{n=k+1}^\infty\alpha_n^2\Big)^{1/2}<\infty.
\end{align}
\end{itemize}
Moreover, under one of the above equivalent conditions,  there exist two numerical constants $c_1, c_2>0$ such that
\begin{align}\label{eqn::universal-constant}
c_1 \cdot Q(\mathfrak{q}; \alpha) \le \Big(\E\Big[\sup_{\xi\in \partial \TT(\mathfrak{q})} | X(\TT(\mathfrak{q}), \alpha; \xi)|^2\Big]\Big)^{1/2}\le \sqrt{\E[M^*(\TT(\mathfrak{q}), \alpha)^2]} \le c_2  \cdot Q( \mathfrak{q}; \alpha).
\end{align}
\end{proposition}

\begin{proof}
In the case where $\mathfrak{q} = (2, 2, 2, \cdots)$, we have $\TT(\mathfrak{q})= \TT_2$ and Proposition \ref{prop-comp-binary} follows immediately from Proposition \ref{prop-eq-binary}. If $\mathfrak{q} = 2^{\mathfrak{l}} = (2^{l_1}, 2^{l_2}, \cdots)$ with $l_i$ positive integers, then by Lemma \ref{lem-dya-subtree}, all the statements in Proposition \ref{prop-comp-binary} is the consequence of Proposition \ref{prop-eq-binary} applied to the sequence $\alpha[\mathfrak{l}]$ defined in \eqref{def-beta}.

Now we deal with the more general case $\mathfrak{q}= (q_1, q_2, \cdots)$ with $q_i\ge 2$. By Lemmas \ref{lem-dominance} and \ref{lem-two-embed}, we have the stochastic dominations:
\[
M^*(\TT(2^{\mathfrak{l}^{-}(\mathfrak{q})}), \alpha)\le_s M^*(\TT(\mathfrak{q}), \alpha)\le_s M^*(\TT(2^{\mathfrak{l}^{+}(\mathfrak{q})}), \alpha)
\]
and
\[
\Sigma_N(\TT(2^{\mathfrak{l}^{-}(\mathfrak{q})}), \alpha)\le_s \Sigma_N(\TT(\mathfrak{q}), \alpha)\le_s \Sigma_N(\TT(2^{\mathfrak{l}^{+}(\mathfrak{q})}), \alpha) \quad \text{for all $N\in\N$.}
\]
Thus the statements in Proposition \ref{prop-comp-binary} for $\TT(\mathfrak{q})$ follows from  those for $\TT(2^{\mathfrak{l}^{\pm}(\mathfrak{q})})$ by the following simple observation: there exist two numerical constants $c_1, c_2>0$ such that
\begin{align}\label{l-pm-alpha}
c_1 Q(\alpha[\mathfrak{l}^{-}(\mathfrak{q})]) \le Q(\alpha[\mathfrak{l}^{+}(\mathfrak{q})]) \le c_2 Q(\alpha[\mathfrak{l}^{-}(\mathfrak{q})])
\end{align}
and
\begin{align}\label{l-alpha}
c_1 Q(\alpha[\mathfrak{l}^{\pm}(\mathfrak{q})]) \le Q(\mathfrak{q}; \alpha) \le c_2 Q(\alpha[\mathfrak{l}^{\pm}(\mathfrak{q})]),
\end{align}
where $Q(\mathfrak{q}; \alpha)$ is defined in \eqref{def-Q-hat}.  We now prove the implication $(4) \Longrightarrow (2)$. Assume that $ Q(\mathfrak{q}; \alpha)<\infty$, then $Q(\alpha[\mathfrak{l}^{+}(\mathfrak{q})])<\infty$. Denote
\[
\mathfrak{l}^{+}(\mathfrak{q}) = (l_1, l_2, \cdots) = \Big(\Big\lceil  \frac{\log q_1}{\log 2} \Big\rceil, \Big\lceil  \frac{\log q_2}{\log 2} \Big\rceil, \dots \Big).
\]
Then by the assumption $q_i\ge 2$, we have  $l_i\ge 1$ for all $i\ge 1$.
Now by the stochastic domination
\[
\Sigma_N(\TT(\mathfrak{q}), \alpha)\le_s \Sigma_N(\TT(2^{\mathfrak{l}^{+}(\mathfrak{q})}), \alpha) \quad \text{for all $N\in\N$}
\]
and the equality of the distribution
\[
  \Sigma_N(\TT(2^{\mathfrak{l}^{+}(\mathfrak{q})}), \alpha)   \stackrel{d}{=}  \Sigma_{K(N)} (\TT_2; \alpha[\mathfrak{l}^{+}(\mathfrak{q})]) \quad \text{with $K(N)=2^{l_1 + \cdots+l_N}$,}
\]
we obtain
\begin{align*}
\E[\Sigma_N(\TT(\mathfrak{q}), \alpha)^2] \le \E [\Sigma_{K(N)} (\TT_2; \alpha[\mathfrak{l}^{+}(\mathfrak{q})])^2] \quad \text{for all $N\in \N$.}
\end{align*}
Now by  \eqref{sigma-N-es}, there exists a numerical constant $C>0$ such that
\begin{align*}
\E [\Sigma_{K(N)} (\TT_2; \alpha[\mathfrak{l}^{+}(\mathfrak{q})])^2]  \le  C \Big[ 2\sqrt{K(N)} Q_{K(N)-1}(\alpha[\mathfrak{l}^{+}(\mathfrak{q})]) + \sum_{l=K(N)}^{\infty}\frac{1}{\sqrt{l+1}} Q_l(\alpha[\mathfrak{l}^{+}(\mathfrak{q})]) \Big]^2.
\end{align*}
By applying similar inequality as \eqref{N-Q-N}, we obtain
\[
\limsup_{N\to\infty}\E[\Sigma_N(\TT(\mathfrak{q}), \alpha)^2] \le \lim_{N\to\infty} \E [\Sigma_{K(N)} (\TT_2; \alpha[\mathfrak{l}^{+}(\mathfrak{q})])^2] = 0.
\]
This is the desired limit equality  \eqref{T-q-unif}  and we complete the proof of the implication $(4) \Longrightarrow (2)$. The remaining part of Proposition \ref{prop-comp-binary} can be similarly proved.
\end{proof}

\subsection{Proofs of Theorem~\ref{thm-bdd-degree} and Theorem~\ref{thm-all-tree}}
Given a rooted tree $\TT$ and a number sequence $\alpha=(\alpha_k)_{k=1}^\infty$.
Define two sequences $\mathfrak{q}=(q_1,q_2,...)$ and $\tilde{\mathfrak{q}}=(\tilde{q}_1,\tilde{q}_2,...)$ as follows:
\begin{align}\label{eqn::def-q}
q_n:=D_{min}^{(n)}(\TT),\quad \tilde{q}_n:=D_{max}^{(n)}(\TT),\quad n=1,2,3,....
\end{align}
We have two rooted trees $\TT(\mathfrak{q})$ and $\TT(\tilde{\mathfrak{q}})$.
There exist natural root-preserving and partial-order preserving isometric embeddings:
\[
\TT(\mathfrak{q})\hookrightarrow\TT\hookrightarrow\TT(\tilde{\mathfrak{q}})
\]
Hence for any $n\in\mathbb{N}$, we have the stochastic dominations:
\begin{align}\label{eqn::M-star}
\begin{split}
&M_n(\TT(\mathfrak{q}), \alpha)\le_s M_n(\TT, \alpha)\le_s M_n(\TT(\tilde{\mathfrak{q}}), \alpha)
\\
& M^*(\TT(\mathfrak{q}), \alpha)\le_s M^*(\TT, \alpha)\leq_s M^*(\TT(\tilde{\mathfrak{q}}), \alpha)
\end{split}
\end{align}
and
\begin{align}\label{eqn::sigma-n}
\Sigma_n(\TT(\mathfrak{q}), \alpha) \le_s \Sigma_n(\TT, \alpha)\le_s \Sigma_n(\TT(\tilde{\mathfrak{q}}), \alpha).
\end{align}
%%%%%%%%%%%%%%%%%%%%%%%%%%%%%%%%%%%%%%%%%%%%%%%%%%%%%%%%%%%%%%%

For a number sequence $\alpha = (\alpha_k)_{k=1}^\infty$, recall
\begin{align}\label{def-Q-alpha}
Q(\alpha): =  \sum_{l=0}^\infty  \Big( \frac{1}{l+1}\sum_{k=l+1}^\infty \alpha_k^2\Big)^{1/2}<\infty.
\end{align}
Then we have the following observation.
\begin{lemma}\label{lem::equiv-q-q}
Suppose that $\mathfrak{q}$ and $\tilde{\mathfrak{q}}$ are defined as \eqref{eqn::def-q} and $Q(\mathfrak{q}; \alpha), Q(\tilde{\mathfrak{q}}; \alpha)$ are defined as \eqref{def-Q-hat}. If the rooted tree $\TT$ satisfies
\eqref{degree-assumption}, then
\[
Q(\alpha)<\infty  \text{ if and only if }Q(\mathfrak{q}; \alpha)<\infty \text{ if and only if }Q(\tilde{\mathfrak{q}}; \alpha)<\infty.
\]
\end{lemma}
%\begin{proof}
%The proof of the lemma is a direct application of the conditions \eqref{degree-assumption}.
%\end{proof}

%%%%%%%%%%%%%%%%%%%%%%%%%%%%%%%%%%%%%%%%%%%%%%%%%%%%%%%%%%%

\begin{proof}[Proof of Theorem \ref{thm-bdd-degree}]
$(1)\Longrightarrow (2)$. If $M^*(\TT,\alpha)<\infty$ a.s., then by \eqref{eqn::M-star},
\[
M^*(\TT(\mathfrak{q}),\alpha)<\infty \,\, a.s.
\]
 By Proposition \ref{prop-comp-binary}, $Q(\mathfrak{q};\alpha)<\infty$ and hence $Q(\tilde{\mathfrak{q}};\alpha)<\infty$ by Lemma \ref{lem::equiv-q-q}. Using Proposition \ref{prop-comp-binary} again, we have
\begin{align*}
\lim_{N\rightarrow\infty}\mathbb{E}[\Sigma_N(\TT(\tilde{\mathfrak{q}}),\alpha)]=\lim_{N\rightarrow\infty}\mathbb{E}\Big[\sup_{n\geq N}\sup_{m\geq 0}\sup_{\xi\in\partial\TT(\tilde{\mathfrak{q}})}\Big|\sum_{k=n}^{n+m}\alpha_kZ(\pi_k(\xi))\Big|^2\Big]=0.
\end{align*}
By \eqref{eqn::sigma-n}, we have
\begin{align*}
\lim_{N\rightarrow\infty}\mathbb{E}[\Sigma_N(\TT,\alpha)]=\lim_{N\rightarrow\infty}\mathbb{E}\Big[\sup_{n\geq N}\sup_{m\geq 0}\sup_{\xi\in\partial\TT}\Big|\sum_{k=n}^{n+m}\alpha_kZ(\pi_k(\xi))\Big|^2\Big]=0.
\end{align*}

$(2)\Longrightarrow (3)$. This step is elementary.

$(3) \Longrightarrow(4)$. Since we have a root-preserving and partial order-preserving isometric embedding $\TT(\mathfrak{q})\hookrightarrow \TT$, the statement (3) for $\TT$ clearly implies the same statement (3) for $\TT(\mathfrak{q})$. Therefore, by Proposition \ref{prop-comp-binary}, we get $Q(\mathfrak{q};\alpha)<\infty$. By Lemma \ref{lem::equiv-q-q}, we have $Q(\alpha)<\infty$.

$(4)\Longrightarrow (1)$. If $Q(\alpha)<\infty$, then by Lemma \ref{lem::equiv-q-q} one has  $Q(\tilde{\mathfrak{q}};\alpha)<\infty$. Using  Proposition~\ref{prop-comp-binary}, we have $M^*(\TT(\tilde{\mathfrak{q}}),\alpha)<\infty$ a.s.  and thus by \eqref{eqn::M-star}, $M^*(\TT,\alpha)<\infty$ a.s.

Finally, \eqref{eqn::equivalent-sup} and \eqref{eqn::equivalent-sup-bis} follow from \eqref{eqn::universal-constant},\eqref{eqn::M-star}, Proposition \ref{prop-comp-binary} and Lemma \ref{lem::equiv-q-q}.
\end{proof}

%%%%%%%%%%%%%%%%%%%%%%%%%%%%%%%%%%%%%%%%%%%%%%%%%%%%%%%%%%%%

\begin{lemma}\label{lem::equiv-q}
Suppose that $\mathfrak{q}$ and $\tilde{\mathfrak{q}}$ are defined as \eqref{eqn::def-q} and $Q(\mathfrak{q}; \alpha), Q(\tilde{\mathfrak{q}}; \alpha)$ are defined as \eqref{def-Q-hat}. If the rooted tree $\TT$ satisfies \eqref{eqn::degree-bound}, then
\[
Q(\mathfrak{q}; \alpha)<\infty \text{ if and only if }Q(\tilde{\mathfrak{q}}; \alpha)<\infty.
\]
\end{lemma}

\begin{proof}[Proof of Theorem \ref{thm-all-tree}]
The proof of Theorem \ref{thm-all-tree} is similar to that of  Theorem \ref{thm-bdd-degree}, we only need to replace Lemma \ref{lem::equiv-q-q} by Lemma \ref{lem::equiv-q}.
\end{proof}

\section{Appendix: Proof of the implications in Remark \ref{Thm::Threeconditions}}
 (c-1)$\Longrightarrow$ (c-2): If $\lim_{k\to\infty}\alpha_k =  0$, then we can write
$
\alpha_k=\sum_{n=k}^{\infty}(\alpha_n-\alpha_{n+1}).
$
By the classical Minkowski's integral inequality, we have
\begin{align*}
\Big(\sum_{k=l}^{\infty}\alpha_k^2\Big)^{1/2}
&=\Big(\sum_{k=l}^{\infty}\Big|\sum_{n=k}^{\infty}(\alpha_n-\alpha_{n+1})\Big|^2\Big)^{1/2}\le \Big(\sum_{k=l}^{\infty}\Big(\sum_{n=l}^{\infty}|\alpha_n-\alpha_{n+1}|\mathds{1}(n\geq k)\Big)^2\Big)^{1/2}
\\
&\leq\sum_{n=l}^{\infty}\Big(\sum_{k=l}^{\infty}\Big(|\alpha_n-\alpha_{n+1}|\mathds{1}(n\geq k)\Big)^2\Big)^{1/2}
 \leq \sum_{n=l}^{\infty}|\alpha_n-\alpha_{n+1}|\sqrt{n-l+1}.
\end{align*}
Using the following elementary inequality,
\[
\sum_{l=1}^n \sqrt{\frac{n-l+1}{l}} \le  \sum_{l=1}^n \sqrt{\frac{n}{l}} \le n,
\]
we obtain
\begin{align*}
Q(\alpha)=&\sum_{l=1}^{\infty}\Big(\frac{1}{l}\sum_{k=l}^{\infty}\alpha_k^2\Big)
^{1/2}
\leq\sum_{l=1}^{\infty}\sum_{n=l}^{\infty}|\alpha_n-\alpha_{n+1}|\sqrt{\frac{n-l+1}{l}}\\
=&\sum_{n=1}^{\infty}|\alpha_n-\alpha_{n+1}|\sum_{l=1}^{n}\sqrt{\frac{n-l+1}{l}}\leq \sum_{n=1}^{\infty}n|\alpha_n-\alpha_{n+1}|.
\end{align*}
Hence the condition (c-1) implies the condition (c-2).

(c-2)$\Longrightarrow$(c-3). Suppose $Q(\alpha)<\infty$, using Cauchy-Schwarz inequality, we have
\begin{align*}
\sum_{k=l}^{\infty}\frac{\alpha_k}{k}\leq \Big(\sum_{k=l}^{\infty}\frac{1}{k^2}\Big)^{1/2}
\Big(\sum_{k=l}^{\infty}\alpha_k^2\Big)^{1/2}\leq \frac{4}{\sqrt{l}}\Big(\sum_{k=l}^{\infty}\alpha_k^2\Big)^{1/2}.
\end{align*}
Then
\begin{align*}
\sum_{k=1}^{\infty}\alpha_k = \sum_{l=1}^{\infty}\sum_{k=l}^{\infty}\frac{\alpha_k}{k}\leq \sum_{l=1}^{\infty}\frac{4}{\sqrt{l}}\Big(\sum_{k=l}^{\infty}\alpha_k^2\Big)^{1/2}=4Q(\alpha)<\infty.
\end{align*}

Finally, suppose that  $(\alpha_k)_{k\ge 1}$ is non-increasing and $\sum_{k=1}^\infty \alpha_k<\infty$.  Then for any $n\in \N$, we have
\begin{align*}
\sum_{k=1}^{n}k|\alpha_k-\alpha_{k+1} &| = \sum_{k=1}^{n}k(\alpha_k-\alpha_{k+1})
=\alpha_1-n\alpha_{n+1}+\sum_{k=2}^{n}\alpha_k
\\
\le &  \alpha_1 + n \alpha_{n+1}+\sum_{k=2}^{n}\alpha_k
\le  \alpha_1 + \sum_{k=2}^{n+1} \alpha_k + \sum_{k=2}^{n}\alpha_k \le 2 \sum_{k=1}^\infty \alpha_k.
\end{align*}
It follows that
\[
\sum_{k=1}^\infty k | \alpha_k - \alpha_{k+1}|\le 2 \sum_{k=1}^\infty \alpha_k<\infty.
\]
Note also that by (c-3), we have $\alpha_k\rightarrow 0$. Hence (c-3) implies (c-1) if $(\alpha_k)_{k\ge 1}$ is non-increasing.
%%%%%%%%%%%%%%%%%%%%%%%%%%%%%%%%%%%%%%%%%%%%%%%%%%%%%%%%%%%
%\bibliographystyle{alpha}
%\bibliography{bib-TREE}

\end{document}